\documentclass[11pt]{amsart}
\usepackage{amssymb,a4wide}

\newcommand{\sig}{\mathfrak c_\sigma}
\newcommand{\bsig}{\mathfrak c_{\bar\sigma}}

\newcommand{\IR}{\mathbb R}
\newcommand{\w}{\omega}
\newcommand{\non}{\mathrm{non}}
\newcommand{\cov}{\mathrm{cov}}
\newcommand{\I}{\mathcal I}
\newcommand{\C}{\mathcal C}
\newcommand{\A}{\mathcal A}
\newcommand{\F}{\mathcal F}
\newcommand{\E}{\mathcal E}
\newcommand{\M}{\mathcal M}
\newcommand{\K}{\mathcal K}
\newcommand{\pr}{\mathrm{pr}}

\newtheorem{theorem}{Theorem}
\newtheorem{lemma}{Lemma}
\newtheorem{problem}{Problem}
\theoremstyle{definition}
\newtheorem{remark}{Remark}

\title{$\sigma$-Continuous functions and related cardinal characteristics of the continuum}

\author{Taras Banakh} 

\begin{document}
\begin{abstract} A function $f:X\to Y$ between topological spaces is called {\em $\sigma$-continuous} (resp. {\em $\bar\sigma$-continuous}) if there exists a  (closed) cover $\{X_n\}_{n\in\w}$ of $X$ such that for every $n\in\w$ the restriction $f{\restriction}X_n$ is continuous. By $\sig$ (resp. $\bsig$) we denote the largest cardinal $\kappa\le\mathfrak c$ such that every function $f:X\to\IR$ defined on a subset $X\subset\IR$ of cardinality $|X|<\kappa$ is $\sigma$-continuous (resp. $\bar\sigma$-continuous). It is clear that $\w_1\le\bsig\le\sig\le\mathfrak c$. We prove that $\mathfrak p\le\mathfrak q_0=\bsig=\min\{\sig,\mathfrak b,\mathfrak q\}\le\sig\le\min\{\non(\mathcal M),\non(\mathcal N)\}$.
\end{abstract}
\address{Ivan Franko National University of Lviv (Ukraine) and Jan Kochanowski University in Kielce (Poland)}
\email{t.o.banakh@gmail.com}
\maketitle

In this paper we introduce and study two cardinal characteristics of the continuum, related to $\sigma$-continuity. 

A function $f:X\to Y$ between two topological spaces is called {\em $\sigma$-continuous} (resp. {\em $\bar\sigma$-continuous}) if there exists a countable (closed) cover $\mathcal C$ of $X$ such that for every $C\in\mathcal C$ the restriction $f{\restriction}C$ is continuous.

The problem of finding a (Borel) function $f:\IR\to\IR$ which is not $\sigma$-continuous is not trivial and was first asked by Lusin. This problem of Lusin was answered by Sierpi\'nski \cite{S36}, \cite{S37} (under CH), Keldy\v s \cite{Kel}, and later by Adyan, Novikov \cite{AN}, van Mill and Pol \cite{vMP}, Jackson, Mauldin \cite{JM}, Cicho\'n, Morayne, Pawlikowski, and Solecki \cite{CMPS}, \cite{Sol}, Darji \cite{Darji}. By a dichotomy of Zapletal \cite{Zap} (generalized by Pawlikowski and Sabok \cite{PS}), a Borel function $f:X\to Y$ between metrizable analytic spaces  is not $\sigma$-continuous if and only if $f$ contains a topological copy of the Pawlikowski function $P:(\w+1)^\w\to \w^\w$ (which is the countable power  of a bijection $\w+1\to\w$). Since the Pawlikowski function $P$ is not $\sigma$-continuous, the family $$\mathcal I_P=\{X\subset(\w+1)^\w:\mbox{$P{\restriction}X$ is not $\sigma$-continuous}\}$$is a proper $\sigma$-ideal of subsets of the compact metrizable space $(\w+1)^\w$. 


Let $\sig$ (resp. $\bsig$) be the largest cardinal $\kappa$ such that every function $f:X\to \IR$ defined on a subset $X\subset\IR$ of cardinality $|X|<\kappa$ is $\sigma$-continuous (resp. $\bar\sigma$-continuous).
 It is clear that $$\w_1\le\bsig\le\sig\le\mathfrak c,$$ so $\bsig$ and $\sig$ are typical small uncountable cardinals in the interval $[\w_1,\mathfrak c]$.

In this paper we establish the relation of the cardinals $\bsig$ and $\sig$ to some known cardinal characteristics of the continuum: $\non(\mathcal M)$, $\non(\mathcal N)$, $\mathfrak p$, $\mathfrak b$, $\mathfrak q_0$, $\mathfrak q$.

Let us define their definitions. By $\w$ we denote the smallest infinite cardinal, by $2^\w$ the Cantor cube $\{0,1\}^\w$, and by $\mathfrak c$ the cardinality of $2^\w$. Let $\mathcal M$ be the ideal of meager sets in the real line, $\mathcal N$ be the ideal of Lebesgue null sets in $\IR$, and $\K_\sigma$ be the $\sigma$-ideal, generated by compact subsets of $\w^\w$. For a family of sets $\mathcal I$ with $\bigcup\I\notin\I$ let $\non(\I)=\min\{|A|:A\subset \bigcup\I\;\;A\notin\I\}$. The cardinal $\non(\K_\sigma)$ is usually denoted by $\mathfrak b$, see \cite{Blass}. 

Let $\mathfrak p$ be the smallest cardinality of a family $\mathcal F$ of infinite subsets of $\w$ such that every finite subfamily $\mathcal E\subset\mathcal F$ has infinite intersection and for any infinite set $I\subset\w$ there exists a set $F\in\mathcal F$ such that $I\setminus F$ is infinite. 

It is known \cite{vD}, \cite{Vau}, \cite{Blass} that $\mathfrak p\le\mathfrak b\le\non(\mathcal M)\le\mathfrak c$, and $\mathfrak p=\mathfrak c$ is equivalent to Martin's Axiom for $\sigma$-centered posets. 

A subset $A\subset \IR$ is called a {\em $Q$-set} if each subset of $A$ is of type $F_\sigma$ in $A$. Let 
\begin{itemize}
\item $\mathfrak q_0$ be the smallest cardinality of a subset $A\subset\IR$, which is not a $Q$-set, and
\item $\mathfrak q$ be the smallest cardinal $\kappa$ such that every subset $X\subset\IR$ of cardinality $|X|\ge\kappa$ is not a $Q$-set.
\end{itemize}
It is known \cite{BMZ} that $$\mathfrak p\le\mathfrak q_0\le\min\{\mathfrak b,\non(\mathcal N),\mathfrak q\}\le\mathfrak q\le\log(\mathfrak c^+),$$ where $\log\mathfrak c^+=\min\{\kappa:2^\kappa>\mathfrak c\}$.
More information on $Q$-sets and the cardinals $\mathfrak q_0$ and $\mathfrak q$ can be found in \cite[\S4]{Miller} and  \cite{BMZ}. The following theorem was proved on 12.09.2019 during the XXXIII International Summer Conference on Real Function Theory in Ustka at the Baltic sea.

\begin{theorem}\label{t:ustka} Let $X,Y$ be separable metrizable spaces. If $|X|<\mathfrak b$, then each Borel function $f:X\to Y$ is $\bar\sigma$-continuous.
\end{theorem}

\begin{proof} Embed the metrizable separable spaces $X,Y$ into Polish spaces $\tilde X$ and $\tilde Y$, respectively.  Let $\{V_n\}_{n\in\w}$ be a countable base of the topology of the Polish space $\tilde Y$. Since the function $f$ is Borel, for every $n\in\w$ the preimage $f^{-1}(V_n)$ is a Borel subset of $X$ and hence $f^{-1}(V_n)=X\cap B_n$ for some Borel set $B_n$ in the Polish space $\tilde X$.

By \cite[13.5]{Ke}, there exists a continuous bijective map $g:\tilde Z\to  \tilde X$ from a Polish space $\tilde Z$  such that for every $n\in\w$ the Borel set $g^{-1}(B_n)$ is open in $\tilde Z$. Consider the subspace $Z=g^{-1}(X)$ of the Polish space $\tilde Z$ and observe that the map $h=f\circ g{\restriction}Z$ is continuous. Indeed, for any $n\in\w$ the preimage $h^{-1}(V_n)$ of the basic open set $V_n$ equals $Z\cap g^{-1}(B_n)$ and hence is open in $Z$. 

Since $|Z|=|g^{-1}(X)|=|X|<\mathfrak b$,  the set $Z=g^{-1}(X)$ is contained in some $\sigma$-compact subset $A$ of $\tilde Z$. Write the set $A$ as the countable union $A=\bigcup_{n\in\w}A_n$ of compact sets $A_n$ in $\tilde Z$. For every $n\in\w$ the image $K_n:=g(A_n)\subset \tilde X$ of $A_n$ is compact and the continuous bijective function $g{\restriction}A_n:A_n\to K_n$ is a homeomorphism and so is the function $g{\restriction}A_n\cap Z:A_n\cap Z\to K_n\cap X$. Then the restriction $f{\restriction}K_n\cap X=f\circ g\circ (g{\restriction}A_n\cap Z)^{-1}$ is continuous, and the function $f$ is $\bar \sigma$-continuous.
\end{proof}

We say that two subsets $A,B$ of a topological space $X$ {\em can be separated by $\sigma$-compact sets} if there are two disjoint $\sigma$-compact subsets $\tilde A$ and $\tilde B$ of $X$ such that $A\subset\tilde A$ and $B\subset\tilde B$. 

\begin{lemma}\label{l:q} The cardinal $\mathfrak q_0$ is equal to the largest cardinal $\kappa$ such that any two disjoint sets $A,B\subset 2^\w$ of cardinality $\max\{|A|,|B|\}<\kappa$ can be separated by $\sigma$-compact sets in $2^\w$.
\end{lemma}

\begin{proof} Let $\mathfrak q_0'$ be  is equal to the largest cardinal $\kappa$ such that any two disjoint sets $A,B\subset 2^\w$ of cardinality $\max\{|A|,|B|\}<\kappa$ can be separated by $\sigma$-compact sets in $2^\w$. Then any set $X\subset 2^\w$ of cardinality $|X|<\mathfrak q_0'$ is a $Q$-set, which implies that $\mathfrak q_0'\le \mathfrak q_0$. Assuming that $\mathfrak q_0'<\mathfrak q_0$, we can find two disjoint sets $A,B\subset 2^\w$ of cardinality $|A\cup B|=\mathfrak q_0'<\mathfrak q_0$ such that $A$ and $B$ cannot be separated by $\sigma$-compact subsets of $2^\w$. We can identify the Cantor cube $2^\w$ with a subspace of the real line. Since $|A\cup B|<\mathfrak q_0$, the subset $X=A\cup B\subset 2^\w\subset\IR$ is a $Q$-set and hence $A$ is an $F_\sigma$-set in $X$. Then $A=K\cap X$ for some $\sigma$-compact set $K\subset 2^\w$. Since the Polish space $P=2^\w\setminus K$ is a continuous image of $\w^\w$, we can use the definition of the cardinal $\mathfrak b\ge\mathfrak q_0>|B|$ and prove that  $B$  is contained in a $\sigma$-compact subset $\Sigma$ of $P$. Then $K$ and $\Sigma$ are disjoint $\sigma$-compact sets separating the sets $A,B$ in $2^\w$, which contradicts the choice of the sets $A,B$. This contradiction shows that $\mathfrak q_0'=\mathfrak q_0$.
\end{proof}


Having in mind the characterization of $\mathfrak q_0$ in Lemma~\ref{l:q}, let us consider two modifications of $\mathfrak q_0$. Namely, let
\begin{itemize}
\item $\mathfrak q_1$ be  the smallest cardinal $\kappa$ for which there exists a subset $A\subset 2^\w$ of cardinality $|A|\le\kappa$ and a family $\mathcal B$ of compact subsets of $2^\w$ with $|\mathcal B|\le\kappa$ such that the sets $A$ and $\bigcup\mathcal B$ are disjoint but cannot be separated by $\sigma$-compact sets in $2^\w$;
\item $\mathfrak q_2$ be the smallest cardinal $\kappa$ for which there exist families $\mathcal A,\mathcal B$ of compact subsets of $2^\w$ with $\max\{|\mathcal A|,|\mathcal B|\}\le\kappa$ such that the sets $\bigcup\mathcal A$ and $\bigcup\mathcal B$ are disjoint but cannot be separated by $\sigma$-compact sets in $2^\w$.
\end{itemize}
This definition and Lemma~\ref{l:q} imply that $$\w_1\le\mathfrak q_2\le\mathfrak q_1\le\mathfrak q_0\le\mathfrak q\le\log(\mathfrak c^+)\le\mathfrak c.$$


The following theorem is the main result of this paper.

\begin{theorem} $\mathfrak p\le \mathfrak q_2\le\mathfrak q_1\le \mathfrak q_0=\bsig=\min\{\sig,\mathfrak b,\mathfrak q\}\le\sig\le\min\{\non(\mathcal M),\non(\mathcal N),\non(\mathcal I_P)\}$.
\end{theorem}

The proof of this theorem is divided into a series of lemmas.

\begin{lemma} $\mathfrak p\le\mathfrak q_2$.
\end{lemma}

\begin{proof} Given any non-empty families $\A,\mathcal B$ of compact sets of $2^\w$ with $\max\{|\mathcal A|,|\mathcal B|\}<\mathfrak p$ and $(\bigcup\A)\cap(\bigcup\mathcal B)=\emptyset$, we shall prove that $\bigcup\A$ and $\bigcup\mathcal B$ can be separated by $\sigma$-compact subsets of $2^\w$. Identify the Cantor cube $2^\w$ with the set of branches of the binary tree $2^{<\w}=\bigcup_{n\in\w}2^n$. For every $n\in\w$ let $\pr_n:2^\w\to 2^n$, $\pr_n:x\mapsto x{\restriction}n$, be the projection of $2^\w$ onto $2^\w$. 

Let $[2^{<\w}]^{<\w}$ be the family of finite subsets of $2^{<\w}$.
For any disjoint compact subset $A,B\subset 2^\w$ and $n\in\w$ consider the set
$$\F_{A,B,n}=\bigcap_{a\in A}\{F\in [2^{<\w}]^{<\w}:F\cap \{a{\restriction}i\}_{i\in\w}=\emptyset\}\cap\bigcap_{b\in B} \{F\in [2^{<\w}]^{<\w}:|F\cap \{b{\restriction}i\}_{i\in\w}|\ge n\}.$$
We claim that this set is infinite. Indeed, since $A,B$ are compact disjoint sets in $2^\w$, there exists $m\in\w$ such that $\pr_m(A)\cap\pr_m(B)=\emptyset$. Then for every $k\ge n$ the set $F_k=\bigcup_{i=m}^{m+k}\pr_i(B)$ belongs to the family $\F_{A,B,n}$.

We claim that the family $$\F=\{\F_{A,B,n}:A\in\A,\;B\in\mathcal B,\;n\in\w\}$$ is infinitely centered in the sense that each finite non-empty subfamily $\mathcal E\subset\F$ has infinite intersection. Indeed, write the family $\E$ as $\E=\{\F_{A_i,B_i,n_i}\}_{i=1}^k$ for some $A_1,\dots,A_k\in\A$, $B_1,\dots,B_k\in\mathcal B$, $n_1,\dots,n_k\in\w$. Consider the compact sets $A=\bigcup_{i=1}^kA_i$ and $B=\bigcup_{i=1}^kB_i$ and observe that they are disjoint as $A\cap B\subset(\bigcup\A)\cap(\bigcup\mathcal B)=\emptyset$. Let $n=\max_{1\le i\le k}n_i$ observe that $\F_{A,B,n}\subset\bigcap\E$, which implies that the intersection $\bigcap\E$ is infinite. 

Since $|\F|=|\w\times\A\times\mathcal B|<\mathfrak p$, the family $\F$ has infinite pseudointersection $\{F_k\}_{k\in\w}\subset [2^{<\w}]^{<\w}$, which means that for every $A\in\A$, $B\in\mathcal B$ and $n\in\w$ the set $\{k\in\w:F_k\notin\F_{A,B,n}\}$ is finite.

For every $n\in\w$ consider the closed subset $$K_n=\bigcap_{k\ge n}\{x\in 2^\w:F_k\cap\{x{\restriction}i\}_{i\in\w}=\emptyset\}=\bigcap_{k\ge n}\bigcap_{i\in\w}\{x\in 2^\w:x{\restriction}i\notin F_k\}$$ of $2^\w$ and observe that $\bigcup\A\subset \bigcup_{n\in\w}K_n\subset 2^\w\setminus \bigcup\mathcal B$. It follows that $2^\w\setminus\bigcup_{n\in\w}K_n$ is a Polish space containing the set $\bigcup\mathcal B$. The definition of the cardinal $\mathfrak b\ge \mathfrak p>|\mathcal B|$ ensures that $\bigcup\mathcal B$ is contained in some $\sigma$-compact subset $\Sigma$ of the Polish space $2^\w\setminus\bigcup_{n\in\w}K_n$. Then $\bigcup_{n\in\w}K_n$ and $\Sigma$ are two disjoint $\sigma$-compact sets separating the sets $\bigcup\A$ and $\bigcup\mathcal B$ in the Cantor cube $2^\w$. 
\end{proof}

\begin{lemma} $\bsig=\mathfrak q_0=\min\{\sig,\mathfrak b,\mathfrak q\}$.
\end{lemma}

\begin{proof} First we prove that $\bsig\le\mathfrak q_0$. Fix a set $X\subset \IR$ of cardinality $|X|=\mathfrak q_0$, which is not a $Q$-set and hence contains a subset $A\subset X$ which is not of type $F_\sigma$ in $X$. Consider the function $f:X\to\{0,1\}\subset\IR$ defined by $f^{-1}(1)=A$ and $f^{-1}(x)=X\setminus A$. Assuming that the function $f$ is $\bar\sigma$-continuous, we would conclude that $A$ and $B$ are $F_\sigma$-sets in $X$, which contradicts the choice of $X$. Consequently, $\bsig\le|X|=\mathfrak q_0$.

Assuming that $\bsig<\mathfrak q_0$, we can find a subset $X\subset\IR$ of cardinality $|X|=\bsig<\mathfrak q_0$ and a function $g:X\to\IR$ which is not $\bar\sigma$-continuous. The strict inequality $|X|<\mathfrak q_0$ implies that $X$ is a $Q$-space. Consequently, each subset of $X$ is of type $F_\sigma$ in $X$ and the function $g$ is Borel. Since $|X|<\mathfrak q_0\le\mathfrak b$, we can apply Theorem~\ref{t:ustka} and conclude that the function $g$ is $\bar\sigma$-continuous, which contradicts the choice of $g$. This contradiction completes the proof of the equality $\bsig=\mathfrak q_0$.
\smallskip

Since $\mathfrak q_0=\bsig\le\sig$, the equality $\mathfrak q_0=\min\{\sig,\mathfrak b,\mathfrak q\}$ will follow as soon as we prove that the strict inequality $\mathfrak q_0<\min\{\mathfrak b,\mathfrak q\}$ implies $\sig\le\mathfrak q_0$.  Assuming that $\mathfrak q_0<\min\{\mathfrak q,\mathfrak b\}$, find a $Q$-set $Y\subset\IR$ of cardinality $|Y|=\mathfrak q_0$. By the definition of $\mathfrak q_0$ there exists a subset $X\subset\IR$ of cardinality $|X|=\mathfrak q_0$ which is not a $Q$-set. Let $f:X\to Y$ be any bijection. We claim that either $f$ or $f^{-1}$ is not $\sigma$-continuous. To derive a contradiction, assume that both maps $f$ and $f^{-1}$ are $\sigma$-continuous. Then there exists a countable cover $\C$ of $X$ such that for every $C\in\C$ the restriction $f{\restriction}C:C\to f(C)$ is a homeomorphism.  By the Lavrentiev Theorem \cite[3.9]{Ke}, for every $C\in\mathcal C$ the topological embedding $f{\restriction}C$ can be extended to a topological embedding $f_C:\tilde C\to\IR$ of some $G_\delta$-subset $\tilde C$ of $\IR$. 
Since $X$ is not a $Q$-set, there exists a subset $A\subset X$ which is not of type $F_\sigma$ in $X$. On the other hand, for every $C\in\C$ the subset $f(C\cap A)$ is of type $F_\sigma$ in the $Q$-space $Y$. Then there exits a $\sigma$-compact set $K_C\subset \IR$ such that $f(C\cap A)=Y\cap K_C$. By the Souslin Theorem \cite[14.2]{Ke}, the Borel set $K_C\cap f_C(\tilde C)$ is a continuous image of $\w^\w$. Using thsi fact and the definition of the cardinal $\mathfrak b>|Y|\ge |f(C\cap A)|$, we can find a $\sigma$-compact set $\Sigma_C\subset K_C\cap f_C(\tilde C)$ that contains the set $f(C\cap A)$. Then $$f_C(C\cap A)=f(C\cap A)\subset \Sigma_C\cap Y\subset f_C(\tilde C)\cap K_C\cap Y\subset K_C\cap Y=f(C\cap A)$$and hence $f_C(C\cap A)=\Sigma_C\cap Y$.
Then $C\cap A=f_C^{-1}(\Sigma_C\cap Y)=f_C^{-1}(\Sigma_C)\cap X$. The continuity of the map $f_C^{-1}:f_C(\tilde C)\to\tilde C$ and the inclusion $\Sigma_C\subset f_C(\tilde C)$ imply that the preimage $\Lambda_C=f_C^{-1}(\Sigma_C)$ is a $\sigma$-compact subset of the Polish space $\tilde C\subset\IR$ such that $\Lambda_C\cap X=f_C^{-1}(\Sigma_C)\cap X=C\cap A$. Then the union $\Lambda=\bigcup_{C\in\C}\Lambda_C$ is a $\sigma$-compact set in $\IR$ such that
$$\Lambda\cap X=\bigcup_{C\in\C}\Lambda_C\cap X=\bigcup_{C\in\C}C\cap A=A,$$
which means that the set $A$ is of type $F_\sigma$ in $X$. But this contradicts the choice of the set $A$. This contradiction shows that one of the maps $f$ or $f^{-1}$ is not $\sigma$-continuous and hence $\sig\le |A|=|B|=\mathfrak q_0$. 
\end{proof}

\begin{lemma} $\sig\le\min\{\non(\mathcal N),\non(\mathcal M),\non(\mathcal I_P)\}$.
\end{lemma}

\begin{proof}
Choose a subset $X\subset(\w+1)^\w$ of cardinality $|X|=\non(\mathcal I_P)$ such that $P{\restriction}X$ is not $\sigma$-continuous. Being a subspace of the zero-dimensional compact metrizable space $(\w+1)^\w$, the space $X$ can be embedded into the real line $\IR$. Consequently, $\sig\le|X|=\non(\mathcal I_P)$.
\smallskip

Let us recall that a Polish space is called {\em perfect} if it has no isolated points. A subset $A$ of a Polish space $X$ is called {\em perfectly meager} if for any perfect Polish subspace $P\subset X$ the intersection $P\cap A$ is meager in $P$. By a result of  Grzegorek \cite{Greg1}, \cite{Greg2} (see also \cite[5.4]{Miller}), the real line $X$ contains a perfectly meager subset $A$ of cardinality $|A|=\mathrm{non}(\mathcal M)$. By the definition of the cardinal $\non(\mathcal M)$, there exists a non-meager set $B\subset\IR$ of cardinality $|B|=\non(\mathcal M)$. Since $|B|=\non(\mathcal M)=|A|$, there exists a  bijective map $f:B\to A$. We claim that one of the maps $f$ or $f^{-1}$ is not $\sigma$-continuous. To derive a contradiction, assume that the maps $f$ and $f^{-1}$ are $\sigma$-continuous. Then we can find a countable cover $\mathcal C$ of $B$ such that for every $C\in\mathcal C$ the restriction $f{\restriction}C$ is a topological embedding. By the Lavrentiev Theorem \cite[3.9]{Ke}, for every $C\in\mathcal C$ the topological embedding $f{\restriction}C$ can be extended to a topological embedding $f_C:\tilde C\to\IR$ of some $G_\delta$-subset $\tilde C$ of $\IR$. Let $U_C$ be the union of open countable subsets in $\tilde C$. The hereditary Lindel\"of property of the space $\tilde C\subset\IR$ implies that the open set $U_C$ is countable and hence meager in $\IR$. On the other hand, the complement $P_C:=\tilde C\setminus U_C$ is a perfect Polish space. Since $f_C$ is a topological embedding, the image $f_C(P_C)$ is a perfect Polish subspace of $\IR$.
The perfect meagerness of $A$ ensures that the intersection $f_C(P_C)\cap A$ is a meager subset of the Polish space $f_C(P_C)$ and then the preimage $(f_C{\restriction}P_C)^{-1}(A)$ is a meager subset of $P_C$ and of $\IR$, too. Then the preimage $ f_C^{-1}(A)\subset U_C\cup (f_C{\restriction}P_C)^{-1}(A)$ is a meager subset of $\IR$ and so is the set $C=(f_C{\restriction}C)^{-1}(A)$. Then the set $B=\bigcup\C$ is meager in $\IR$, which contradicts the choice of $B$. This contradiction shows that one of the maps $f$ or $f^{-1}$ is not $\sigma$-continuous and hence $\sig\le |A|=|B|=\mathrm{non}(\mathcal M)$.
\smallskip

A subset $A\subset\mathbb R$ has {\em universal measure zero} if $\mu(A)=0$ for any Borel continuous probability measure $\mu$ on $\mathbb R$. By a result of Grzegorek and Ryll-Nardzewski \cite{GR} (see also \cite[5.4]{Miller}), the real line contains a subset $A$ of universal measure zero that has cardinality $|A|=\mathrm{non}(\mathcal N)$.
 By the definition of the cardinal $\non(\mathcal N)$, there exists a subset $B\notin \mathcal N$ of $\IR$ with $|B|=\non(\mathcal N)$.  Since $|B|=\non(\mathcal N)=|A|$, there exists a  bijective map  $g:B\to A$. We claim that the map $g$ is not $\sigma$-continuous. To derive a contradiction, assume that $g$ is $\sigma$-continuous and find a countable cover $\mathcal C$ of $B$ such that for every $C\in\mathcal C$ the restriction $f{\restriction}C$ is continuous. Dividing each set $C\in\mathcal C$ into countably many pieces, we can assume that $C$ is bounded in the real line.
 
 By the Kuratowski Theorem \cite[3.8]{Ke}, for every $C\in\mathcal C$ the continuous map $g{\restriction}C$ can be extended to a continuous map $g_C:\tilde C\to\IR$ defined on some $G_\delta$-subset $\tilde C\subset \bar C$ of $\IR$. Let $\lambda$ be the Lebesgue measure on $\IR$. For every $C\in\C$ consider the subset  $A'_C=\{a\in A:\lambda(g_C^{-1}(a))>0\}$ and observe that it is at most countable. Then the set $A'=\bigcup_{C\in\C}A_C'$ is at most countable, too. Let $P=\bigcup_{C\in\C}g_C^{-1}(A')$ and $\C'=\{C\in\C:\lambda(\tilde C\setminus P)>0\}$. For every $C\in\C'$ the set $\tilde C\setminus P$ has finite non-zero measure. So, we can define  a Borel probability measure $\mu_C$ on $\IR$ letting $\mu_C(X)={\lambda(g_C^{-1}(X)\setminus P)}/{\lambda(\tilde C\setminus P)}$ for any Borel subset $X\subset\IR$. It is easy to see that the measure $\mu_C$ is continuous (which means that the measure of each singleton is zero). Since the set $A$ has universal measure zero, $$0=\mu_C(A)=\lambda(g_C^{-1}(A)\setminus P)=\lambda(B\cap \tilde C\setminus P).$$ Then $\lambda(B\setminus P)\le\sum_{C\in\C}\mu(B\cap C\setminus P)=0$  and $$\lambda(B)=\lambda(B\cap P)+\lambda(B\setminus P)=0+0=0$$as $B\cap P=g^{-1}(A')$ is at most countable. But the equality $\lambda(B)=0$ contradicts the choice of $B$. This contradiction shows that the map $g$ is not $\sigma$-continuous and hence $\sig\le|B|=\mathrm{non}(\mathcal N)$.
\end{proof}

\begin{problem} Which of the strict inequalities $$\bsig<\sig,\;\mathfrak p<\mathfrak q_2,\;\mathfrak q_2<\mathfrak q_1,\;\mathfrak q_1<\mathfrak q_0$$ is consistent with ZFC?
\end{problem}

\begin{problem} Is the strict inequality $\mathfrak q_0<\min\{\mathfrak q,\mathfrak b\}$ consistent?
\end{problem}

\begin{problem} Is $\sig\le\mathfrak b$?
\end{problem}

\begin{remark} Let $\kappa$ be a cardinal. A subset $A$ of a topological space $X$ is called a {\em $G_{<\kappa}$-set} if $A$ can be written as the intersection $A=\bigcap\mathcal A$ of some family $\A$ of open sets in $X$ with $|\mathcal A|<\kappa$. So, $G_{\w_1}$-sets are exactly $G_\delta$-sets.

For a cardinal $\kappa$ let $\mathfrak b_\kappa$ be the smallest cardinality of a subset $A$ of a $G_{<\kappa}$-set $G$ in $\IR$ such that $A$ is not contained in a $\sigma$-compact subset of $G$. It can be shown that $\mathfrak b_\kappa$ is a non-increasing function on the variable $\kappa$ such that
$$\mathfrak b_{\w_1}=\mathfrak b,\;\;\mathfrak b_{\mathfrak c^+}=\w_1,\;\;
\mathfrak b_{\mathfrak p}\ge\mathfrak p,\;\;\mathfrak b_{\mathfrak q_1}\ge\mathfrak q_1,\;\;\mbox{and}\;\;\mathfrak q_1=\max\{\kappa:\kappa\le\mathfrak b_\kappa\}.$$ 
\end{remark}

By \cite{Step}, for every $A\in\mathcal I_P$ the image $P(A)$ is a meager subset of $\w^\w$. This implies that $\non(\mathcal I_P)\le\non(\mathcal M)$ and $\cov(\mathcal I_P)\ge\cov(\mathcal M)$. By \cite{Step}, it is consistent that $\cov(\mathcal I_P)>\cov(\mathcal M)$.

\begin{problem} Is the strict inequality $\non(\mathcal I_P)<\non(\M)$ consistent?
\end{problem} 

\end{document}